\documentclass[11pt,a4paper]{article}

\usepackage{a4wide}
\usepackage{graphicx}
\usepackage{latexsym}
\usepackage{epsfig}
\usepackage{amssymb}
\usepackage{amstext}
\usepackage{amsgen}
\usepackage{amsxtra}
\usepackage{amsgen}
\usepackage{amsthm}

\newtheorem{thm}{Theorem}[section]
\newtheorem{prop}[thm]{Proposition}

\newtheorem{cor}[thm]{Corollary}

\theoremstyle{definition}
\newtheorem{example}[thm]{Example}

\numberwithin{equation}{section}

\newcommand{\R}{\mathbb{R}}

\begin{document}
\title{Kinetic Equations for Processes on Co-evolving Networks}
\author{Martin Burger\thanks{Department Mathematik,  Friedrich-Alexander Universit\"at Erlangen-N\"urnberg, Cauerstr. 11, D 91058 Erlangen, Germany. e-mail: martin.burger@fau.de  } }
\maketitle
\begin{abstract}
The aim of this paper is to derive macroscopic equations for processes on large co-evolving networks, examples being  opinion polarization with the emergence of filter bubbles or other social processes such as norm development. This leads to processes on graphs (or networks), where both the states of particles in nodes as well as the weights between them are updated in time. 
In our derivation we follow the basic paradigm of statistical mechanics: We start from paradigmatic microscopic models and derive a Liouville-type equation in a high-dimensional space including not only the node states in the network (corresponding to positions in mechanics), but also the edge weights between them. We then derive a natural (finite) marginal hierarchy and pass to an infinite limit. 

We will discuss the closure problem for this hierarchy and see that a simple mean-field solution can only arise if the weight distributions between nodes of equal states are concentrated. In a more interesting general case we propose a suitable closure at the level of a two-particle distribution (including the weight between them) and discuss some properties of the arising kinetic equations. Moreover, we highlight some structure-preserving properties of this closure and discuss its analysis in a minimal model. We discuss the application of our theory to some agent-based models in literature and discuss some open mathematical issues. 
%{\bf Keywords: }  
%%{\bf AMS Subject Classification: }  74N20, 74N25, 35K50, 35K55, 65M60.
%
%%{\bf PACS: } 02.30.Jr, 02.30.Xx, 02.60.Lj, 81.10.Aj, 64.70.Nd
\end{abstract}

\section{Introduction}

Following the seminal work by Boltzmann and Maxwell (cf. \cite{boltzmann,maxwell}), kinetic equations have
emerged as a standard tool for the description of (stochastic) interacting particle systems.
Nowadays their rigorous mathematical treatment as well as the derivation of macroscopic
models (cf. \cite{bouchut,cercignani,cercignani2,cercignani3,glassey,villani}) are reasonably well understood. In many modern applications of interacting
particle systems, in particular in social and biological systems, there is a key ingredient
not included in the basic assumptions of kinetic theory, namely a dynamic network structure
between particles (rather called agents in such systems and thus used synonymously in this paper).
While models in physics assume that interactions happen if particles are spatially close, social
interactions rather follow a network structure between particles, which changes at the same
time as the state of the particles (thus called co-evolution, cf. \cite{thurner} and references therein). A canonical example is
the formation of opinions or norms on social networks, where interactions can only happen if
the agents are connected in the network. On the other hand the network links are constantly
changing, and these processes influence each other: Opinions change if there is a connection,
e.g. for followers of some posts. Vice versa agents may tend to follow others with similar
opinions.

In this paper we thus want to establish an approach towards kinetic equations for interacting
particle systems with co-evolving network structures. We consider processes where each
agent has a state that can change in an interaction and there is a network weight between
two agents (zero if not connected), which can change over time. This includes a variety of
microscopic agent-based models in recent literature. We start from a Liouville-type equations
that describes the evoluion of the joint probability measure of the N agent states as well as the
$N(N - 1)$ weights between agents (respectively $\frac{N(N-1)}2$ in the case of undirected networks).
From those we derive a moment hierarchy resembling the original BBGKY-hierarchy, with
the difference that here we derive equations for the probability measure describing k agent
states and $k(k - 1)$ weights between them ( $\frac{k(k-1)}2$ in the undirected case).
It turns out that a key difference to standard types of kinetic models without co-evolving
networks (weights) is the problem to find a simple closure relation. Since the one-particle
distribution does not depend on weights at all, there is obviously no solution of the hierarchy
in the form of product measures (as the classical Stosszahlansatz by Boltzmann). We show
that such a solution exists only in the case of special solutions that exhibit concentration
in the weight variables. In the general case we propose a closure relation at the level of
the two-particle and single-weight distribution. We discuss the mathematical properties of
the resulting equations as well as the existence of different types of stationary solutions,
which are relevant for processes on social networks (such as opinion formation of social norm
construction).
In several examples we discuss how the kinetic equations and some special forms relate
to agent-based models recently introduced in literature. Among others these microscopic models have been used to simulate the following issues:
\begin{itemize}

\item Opinion formation on social networks, including polarization and the formation of echo chambers 
(cf. \cite{baumann,benatti,boschi,chitra,gu,maia,nigam,sugishita})

\item Knowledge networks (cf. \cite{tur}), with state $s$ being a degree of knowledge.

\item Social norm formation and social fragmentation (cf. \cite{kohne,pham,pham2})

\item Biological transport networks (cf. \cite{hu,albi}), with state corresponding to a pressure or similar variable and the weight encoding the capacity of network links.

\end{itemize} 

The derivation of macroscopic models in this context is not just a mathematical exercise, but appears to be of high relevance in order to obtain structured predictions about pattern formation in such systems. While agent-based models rely on simulations for special parameter sets, which can hardly be calibrated from empirical data, the analysis of macroscopic models can provide explanations and predictions of patterns and transitions obtained at the collective level. In this way some concerns about agent-based models like their reproducibility and their limitation to special parameter values (cf. \cite{carney,conte,donkin}) can be avoided. 

The paper is structured as follows: in Section 2 we discuss a paradigmatic model based on a Vlasov-type dynamics, which allows to highlight the basic ideas and properties of the microscopic models, in particular the description via a distribution of $N$ particles and $N(N-1)$ weights. Section 3 presents a general structure for microscopic models and discusses several special cases with examples from agent-based simulations in literature. In Section 4 we discuss the hierarchy obtained from the moments of the microscopic distribution and its infinite limit. We highlight the non-availability of a closed-form solution in terms of a single particle distribution and the need to describe the system in terms of a pair distribution (the distribution for two  particles and the weight between them). As a consequence we also discuss possible closure relations at the level of the pair distribution. These closure relations are further investigated for a minimal model with binary states and weights in Section 5. At this level we can also identify simple structural assumptions for the formation of polarization patterns. Section 6 is devoted to a more detailed mathematical study of the macroscopic version of the paradigmatic model from Section 2, with a particular focus on a closure relation based on the conditional distribution. In Section 7 we further discuss some modelling issues like social balance theory and related processes, which effectively lead to triplet interactions
in weights or states.   Finally, we also discuss a variety of
open and challenging mathematical problems for the equations at the level of  pair distributions.

\section{A Paradigmatic Model: Vlasov-type Dynamics}

In this section and for the exposition of further arguments we consider the genuine system 
\begin{align}
\frac{ds_i}{dt} &= \frac{1}N \sum_{j \neq i} U(s_i,s_j,w_{ij}) \label{eq:microscopic1} \\
\frac{dw_{ij}}{dt} &= V(s_i,s_j,w_{ij}) \label{eq:microscopic2}
\end{align}
in order to highlight the mathematical properties and the derivation of macroscopic equations.
In a similar spirit as \cite{thurner} we will use the continuous time model as a paradigm, but in an analogous way we will also consider other types of interactions leading to kinetic equations in the next section. Here $s_i \in \R^m$ denotes the state variable 
and $w_{ij} \in \R$ the weight between node $i$ and $j$. This minimal model naturally encodes the typical processes and network co-evolutions as also proposed in \cite{thurner}. The interactions of states are mitigated by the weight on the edge between them, while the change of weights on an edge depends on the states of the vertex it connects. Note that we have incorporated a mean-field scaling already in the above system, other types of scaling are left for future research. Let us mention that \eqref{eq:microscopic1}, \eqref{eq:microscopic2} shares some similarities with the interaction models with time varying weights, which have been analyzed in detail in \cite{ayi,mcquade,pouradier}. In their case the weight is independent of $j$ however, which corresponds to a special solution where $V$ is independent of $s_j$ and the $w_{ij}$ have the same initial value for all $j$.

In many cases it is desirable to have symmetry of the network (coresponding to an undirected graph), i.e. $w_{ij} = w_{ji}$ for $i \neq j$, and no loops, i.e. $w_{ii}=0$. This is preserved by a natural symmetry condition for $V$, namely
\begin{equation} \label{eq:Vsymmetry}
V(s,\sigma,w) = V(\sigma,s,w) \qquad \forall s,\sigma \in \R^m, w \in \R.
\end{equation}
 We shall assume that $U$ and $V$ are Lipschitz-continuous functions, which directly implies the existence and uniqueness for \eqref{eq:microscopic1}, \eqref{eq:microscopic2} by the Picard-Lindel\"of Theorem. 

\begin{prop}
Let $U: \R^m \times \R^m \times \R \rightarrow \R^m$ and $V: \R^m \times \R^m \times \R \rightarrow \R$ be Lipschitz-continuous functions.   Then there exists a unique solution of the initial-value problem for \eqref{eq:microscopic1}, \eqref{eq:microscopic2} such that $s_i \in C^1(\R_+)$, $w_{ij} \in C^1(\R_+)$ . If $w_{ij}(0) = w_{ji}(0)$ for all $i \neq j$ and \eqref{eq:Vsymmetry} is satisfied, then
$w_{ij}(t) = w_{ji}(t)$ for $i \neq j$ and all $t \in \R_+$.
\end{prop}

Let us mention a canonical example of the interactions, namely
\begin{equation} \label{eq:Uexample}
U(s_i,s_j,w_{ij})  = - w_{ij} K(s_i-s_j)
\end{equation} 
with an odd kernel $K$ (e.g. $K=\nabla G$ for an even and attractive potential) and
\begin{equation} \label{eq:Vexample}
V(s_i,s_j,w_{ij}) = \eta(s_i-s_j) - \kappa  w_{ij} 
\end{equation}
with nonnegative kernels $\eta$ and $\kappa$. Here $w_{ij}$ is directly the interaction strength, the weight is increased for states $s_i$ and $s_j$ close and decays with relaxation time $\kappa^{-1}$. If $\eta = - c G$ for some constant $c > 0$, then the model \eqref{eq:microscopic1}, \eqref{eq:microscopic2} has a gradient structure of the form
\begin{equation} \label{eq:microscopicgf}
 \frac{ds_i}{dt} = - \nabla_{s_i} E^N, \qquad \frac{dw_{ij}}{dt} = - 2 c N \partial_{w_{ij}} E^N \end{equation}
with the microscopic energy functional
\begin{equation}
 E^N(s_1,\ldots,s_N,w_{12},\ldots,w_{N,N-1}) = \frac{1}{2N} \sum_{i=1}^N \sum_{j\neq i} \left( w_{ij} G(s_i-s_j) + \frac{\kappa w_{ij}^2}{2c}\right).
\end{equation}
A more general version of gradient flows is obtained with \eqref{eq:microscopicgf} and the more general energy functional
\begin{equation}
 E^N(s_1,\ldots,s_N,w_{12},\ldots,w_{N,N-1}) = \frac{1}{2N} \sum_{i=1}^N \sum_{j\neq i}  F(s_i,s_j,w_{ij} ) .
\end{equation}
This implies that the forces are derived from the potential $F$ via
\begin{equation} \label{eq:potential}
 U(s,\sigma,w) = - \nabla_{s} F(s,\sigma,w), \qquad    V(s,\sigma,w) = - c \partial_{w}  F(s,\sigma,w). 
\end{equation}

\subsection{The $N$-particle and weight measure}

In classical kinetic theory the evolution of the system is first described by the joint measure of the $N$ particles, which corresponds to $s_1, \ldots, s_N$. In our setting we need to extend this to a joint measure of the $N$ states $s_i$ and the $N(N-1)$ weights $w_{ij}$, which we denote by $\mu^N_t$ and still denote as $N$-particle measure. 
The corresponding Liouville-type equation for $\mu^N_t$ is given by 
\begin{equation} \label{eq:liouville}
\partial_t \mu^N_t + \frac{1}N  \sum_i  \sum_{j \neq i} \nabla_{s_i} \cdot (  U(s_i,s_j,w_{ij}) \mu_t^N ) 
+  \sum_i  \sum_{j \neq i}  \partial_{w_{ij}} \cdot ( V(s_i,s_j,w_{ij}) \mu_t^N )  = 0.
\end{equation} 
The weak formulation of \eqref{eq:liouville} is given by
\begin{align} 
\frac{d}{dt} \int \varphi(z_N) \mu_t^N(dz_N)  =&  \sum_i  \sum_{j \neq i} \int ( \frac{1}N \nabla_{s_i} \varphi(z_N)    U(s_i,s_j,w_{ij}) + 
\nonumber \\ & \qquad \qquad
\partial_{w_{ij}}  \varphi(z_N)   V(s_i,s_j,w_{ij}) )\mu_t^N(dz_N) , \label{eq:liouvilleweak} 
\end{align}
for $\varphi \in {\cal S}(\R^{N(m+1)})$. 
For Lipschitz-continuous velocities $U$ and $V$ the existence and uniqueness of a weak solution can be modified in a straight-forward way by the method of characteristics (cf. \cite{golse}), the solution $\mu_t^N$ is the push-forward of $\mu_0^N$ under the unique solutions $(s_i,w_{ij})$ of \eqref{eq:microscopic1}, \eqref{eq:microscopic2}.

In the case of forces derived from a potential $F$, i.e. \eqref{eq:potential}, the energy is given by
%$$ U(s,\sigma,w) = - \nabla_{s} F(s,\sigma,w), \qquad    V(s,\sigma,w) = - \nabla_{w}  V(s,\sigma,w), $$
 $$ E[\mu_t^N] = \frac{1}{2N} \sum_{i=1}^N \sum_{j\neq i} \int F(s_i,s_j,w_{ij}) ~\mu_t^N(dz_N) , $$
the gradient flow structure is propagated to the Liouville equation in an Otto-type geometry (cf. \cite{jko,otto}) as
\begin{align*}
\partial_t \mu^N_t &= \sum_i  \sum_{j \neq i} \nabla_{s_i} \cdot (  \mu_t^N \nabla_{s_i} E') 
+  2c N \sum_i  \sum_{j \neq i}  \partial_{w_{ij}}  ( \partial_{w_{ij}}E(s_i,s_j,w_{ij}) \mu_t^N   ) .
\end{align*}
In particular we obtain an energy dissipation of the form
$$ \frac{d}{dt} E[\mu_t^N] = -  \sum_i  \sum_{j \neq i} \int \left( \vert\nabla_{s_i} F(s_i,s_j,w_{ij})\vert^2 + c (\partial_{w_{ij}}F(s_i,s_j,w_{ij}))^2 \right)~\mu_t^N(dz_N).$$
We will later return to the gradient flow structure and dissipation in the context of macroscopic equations and investigate their possible preservation. 

%
%\subsubsection{Mean Weight}
%The evolution of the mean value of the weight $w_{ij}$ is determined by the equation
%\begin{equation}
%\frac{d}{dt} \overline{w}_{ij} = \frac{d}{dt} \int w_{ij} \mu_t^N(dz_N) = \int V(s_i,s_j,w_{ij})  \mu_t^N(dz_N) .
%\end{equation}
%
%\begin{align*} 
%\frac{1}2 \frac{d}{dt}  \int (w_{ij} - \overline{w}_{ij})^2 \mu_t^N(dz_N) =&
%\int (w_{ij} - \overline{w}_{ij})(V(s_i,s_j,w_{ij}) - V(s_i,s_j,\overline{w}_{ij})) \mu_t^N(dz_N) + \\
%& \int (w_{ij} - \overline{w}_{ij}) V(s_i,s_j,\overline{w}_{ij}) \mu_t^N(dz_N)
%\end{align*}

\subsection{Time Scales}

Concerning time we can investigate different scaling limits, which effectively mean a relative scaling of the forces $U$ and $V$. We also discuss a third (mixed) case, which relates to the concept of network-structured models in the limit. 

\subsubsection{Instantaneous Network Formation} 

In some models the network is rebuilt in very small time scales. An extreme case with instantaneous network formation are early bounded confidence models of opinion formation such as the Hegselmann-Krause \cite{hegselmann} or
Deffuant-Weissbuch model  \cite{deffuant}, where the network is built in each time steps between all agents having opinions inside a certain confidence interval. In order to model such a situation it is convenient to introduce a small parameter $\epsilon > 0$ in \eqref{eq:microscopic2} and instead considered the scaled equation. 
\begin{equation}
\epsilon \frac{dw_{ij}}{dt}  = V(s_i,s_j,w_{ij}) \label{eq:microscopic2a}
\end{equation} 
The corresponding Liouville equation is given by
\begin{equation} \label{eq:liouvilleeps}
\partial_t \mu^N_t + \frac{1}N  \sum_i  \sum_{j \neq i} \nabla_{s_i} \cdot (  U(s_i,s_j,w_{ij}) \mu_t^N ) 
+  \frac{1}\epsilon \sum_i  \sum_{j \neq i}  \partial_{w_{ij}} \cdot (  V(s_i,s_j,w_{ij}) \mu_t^N )  = 0.
\end{equation} 

In the simplest case there exists a unique solution $\omega(s,\sigma)$ of the equation
\begin{equation} \label{eq:omega}
V(s,\sigma,\omega(s,\sigma))  = 0 
\end{equation} 
and we expect convergence to the reduced equation
\begin{equation}
\partial_t \mu^N_t + \frac{1}N  \sum_i  \sum_{j \neq i} \nabla_{s_i} \cdot (  U(s_i,s_j,\omega(s_i,s_j)) \mu_t^N ) = 0
\end{equation} 
under suitable properties of $V$. This equation is in the standard form of interaction equations for the particles $s_i$ that can be described by a mean-field limit (cf. \cite{golse}).  

%More precisely, we find fast concentration in the $w$-variable if $-V$ is a monotone function of $w$, i.e. 
%\begin{equation} \label{eq:Vassumption}
%- ( V(s,\sigma,w) - V(s,\sigma,\tilde w) ) \cdot (w- \tilde w) \geq \gamma(s,\sigma) |w-\tilde w|^2
%\end{equation} 
%for $\gamma$ positive. 
 %Here we only give a formal argument based on second moments, assuming suitable smoothness we find
%\begin{align*}
%&\frac{1}2 \frac{d}{dt}  \sum_i \sum_{j \neq i} \int |w_{ij} - \omega(s_i,s_j)|^2 \mu_t^N(dz_N) =   \\
%& \qquad \qquad  \sum_i  \sum_{j \neq i} \int  \frac{1}N  (w_{ij} - \omega(s_i,s_j)) \nabla_{s_i} \omega(s_i,s_j)    U(s_i,s_j,w_{ij})\mu_t^N(dz_N) +\\
%& \qquad \qquad  \frac{1}\epsilon \sum_i  \sum_{j \neq i}\int (w_{ij} - \omega(s_i,s_j))  ( V(s_i,s_j,w_{ij}) - V(s_i,s_j,\omega(s_i,s_j))\mu_t^N(dz_N) 
%\end{align*}
%and with \eqref{eq:Vassumption} we have
%\begin{align*}
%\frac{1}2 \frac{d}{dt}  \sum_i \sum_{j \neq i} \int |w_{ij} - \omega(s_i,s_j)|^2 \mu_t^N(dz_N) \leq - 
  %\frac{C}{2\epsilon} \sum_i \sum_{j \neq i} \int |w_{ij} - \omega(s_i,s_j)|^2 \mu_t^N(dz_N) 
%\end{align*}
%where
%$$ C = 2 \inf_{s,\sigma} \gamma(s,\sigma). $$

Note that we can also study the problem at a fast time scale (rescaled from $t$ to $\epsilon^{-1}t$), which corresponds to \eqref{eq:microscopic2} coupled with 
$$\frac{ds_i}{dt}  = \epsilon \frac{1}N \sum_{j \neq i} U(s_i,s_j,w_{ij}). $$
%\label{eq:microscopic1} \\
%\frac{dw_{ij}}{dt} &= V(s_i,s_j,w_{ij}) \label{eq:microscopic2}
Apparently the limit is given by stationary states $s$ at this scale and thus \eqref{eq:microscopic2} is a pure network adaption model in a given environment
$$ \frac{dw_{ij}}{dt}(t)  = V(s_i,s_j,w_{ij}(t)) = V_{ij}(w_{ij}(t). $$

\subsubsection{Instantaneous State Adaption}

The opposite time scale is related to  a fast adaptation of states instead of the weights. This amounts to \eqref{eq:microscopic2} coupled with
\begin{equation}
\epsilon \frac{ds_i}{dt}  =  \frac{1}N \sum_{j \neq i} U(s_i,s_j,w_{ij}).
\end{equation} 
The limit $\epsilon \rightarrow 0$ is much more complicated in this case compared to the instantaneous network adaption. Formally it is described by $\sum_{j \neq i} U(s_i,s_j,w_{ij})=0$, which is a fully coupled system among the states and weights.

At a fast time scale we obtain instead stationarity of the weights $w_{ij}$ and hence the state adaption is a standard interacting particle system (in a heterogenous environement however)
\begin{equation}
 \frac{ds_i}{dt}(t)  =  \frac{1}N \sum_{j \neq i} U(s_i(t),s_j(t),w_{ij}) = 
\frac{1}N \sum_{j \neq i} U_{ij}(s_i(t),s_j(t) ).
\end{equation}
The mathematical complication in a limit of this system for large $N$ is  inherent in the fact that the particles are not indistinguishable due to the specific forces $U_{ij}$.

\subsubsection{Network-Structured Models} 

Let us finally consider a case where the state is composed of two distinct parts $s_i=(x_i,y_i)$ such that
$V$ depends only on the first part $x_i$. If we now consider a fast adaption in the $x$ part as well as the weights the effective limit will provide stationary weights $w_{ij} = \omega(x_i,x_j)$ and the remaining equation  for the $y$ components becomes (with $U^Y$ the corresponding component of the force $U$)
\begin{equation}
 \frac{dy_i}{dt}(t)  =  \frac{1}N \sum_{j \neq i} U^Y(x_i,y_i(t),x_j,y_j(t),\omega(x_i,x_j))  = 
  \frac{1}N \sum_{j \neq i} \tilde U(x_i,x_j;y_i(t),y_j(t)).
\end{equation}
This corresponds to the framework of network-structured models derived in \cite{burger2}.

\section{Microscopic Models for Processes on Co-Evolving Networks}
 
We consider finite weighted graphs ${\cal G}^N(t) =({\cal V},{\cal E}(t),w(t))$, $t \in \R_+$ with $N$ vertices and symmetric edge weights $w_{ij}(t)$ evolving in time.  Moreover, we consider a vertex function $s:{\cal V} \times \R_+ \rightarrow \R^m$. The value $s_i(t)$ describes the state of agent $i$ (identified with the respective vertex) at time $t$.  Setting $w_{ij}(t) = 0 $ for $(i,j) \notin {\cal E}(t)$, we can equivalently describe such a graph by $(s,w) \in {\cal S}^N \times {\cal W}^N$, where ${\cal S}^N$ is a subset of $\R^{mN}$ and ${\cal W}^N$ is a subset of $\R^{N \times N}$. The exact shape of ${\cal W}^N$ depends on the specific model, e.g. we can restrict to nonnegative weights or include the case of an unweighted graph by choosing ${\cal W}^N = \{0,1\}^{N \times N}$.
We mention that for undirected graphs we can use symmetry of $w$ and avoid loops by $w_{ii}=0$, in this case the description could be even reduced to an element in $\R^{mN + N(N-1)/2}$. Starting from the weights $w$ we say that there is an edge or link between $i$ and $j$ if $w_{ij} \neq 0$.

The microscopic description of the process and the co-evolving network can be carried out by using a probability measure $\mu^N_t$ on ${\cal S}^N \times {\cal W}^N$ for $t \in \R_+$ and formulating an evolution equation for $\mu^N$.  We will now first derive a general model structure and then pass to some other special cases based on examples from literature.  

\subsection{General Model Structure}

In the following we provide a general structure for the microscopic kinetic models, including continuum structures as in the paradigmatic model above, but alos operators arising from long-range jumps (e.g. abrupt changes of weights or states) and discrete models such as binary 
${\cal W} = \{0,1\}$. In addition we consider a random (diffusive) change of states or weights, possibly again depending on the variables themselves, which yields
\begin{align}
\partial_t \mu_t^N =& \frac{1}N \sum_{i=1}^N \sum_{j\neq i} D_{s_i}^* (U(s_i,s_j,w_{ij}) \mu_t^N) + \sum_{i=1}^N \sum_{j\neq i} D_{w_{ij}}^* (V(s_i,s_j,w_{ij}) \mu_t^N) \nonumber \\
& + \sum_{i=1}^N  D_{s_i}^* (U_0(s_i) \mu_t^N) - \sum_{i=1}^N  D_{s_i}^*  D_{s_i} (Q(s_i) \mu_t^N) - \sum_{i=1}^N \sum_{j\neq i} D_{w_{ij}}^*D_{w_{ij}}(R(s_i,s_j,w_{ij}) \mu_t^N). \label{eq:general} 
\end{align}
A variety of different models of the above kind can be found in literature, below we will show how to put them into this general structure.

\subsection{Important Cases}

\subsubsection{Continuous State and Weight Model with State Diffusion }

A slight variation of the paradigmatic model is obtained when adding random change of the state variable via a Wiener process (with effective diffusion coefficient $R$), which leads to 
\begin{align}
&\partial_t \mu_t^N + \frac{1}N \sum_{i=1}^N \sum_{j\neq i} \nabla_{s_i} \cdot (U(s_i,s_j,w_{ij}) \mu_t^N) + \sum_{i=1}^N \sum_{j\neq i} \partial_{w_{ij}}  (V(s_i,s_j,w_{ij}) \mu_t^N) \nonumber \\
&  \qquad  = -\sum_{i=1}^N  \nabla_{s_i} \cdot (U_0(s_i) \mu_t^N)  +\sum_{i=1}^N  \Delta_{s_i} (Q(s_i) \mu_t^N)  .
\end{align}
Here $D_{s_i} = \nabla$ and $D_{s_i}^*= - \nabla \cdot$. In particular in opinion formation models the impact of noise in the state variable has been highlighted in the last years (cf. \cite{carro,raducha}), which is then incorporated in the additional diffusion term. We give a recent example of an agent-based model combining noise with a network co-evolution:

\begin{example}
In  \cite{boschi} a completely continuous model is studied, given by (with the rescaled weights $w_{ij} = N J_{ij}$)
\begin{align*}
ds_i &= - s_i dt + I_i dt + \frac{1}N \sum_{j \neq i} w_{ij} g(s_j) dt + \sigma dW_i \\
dw_{ij} &= \gamma ( J_0 g(s_i) g(s_j) - w_{ij}) dt
\end{align*}
where the $W_i$ are uncorrelated Wiener processes, $\gamma$ and $J_0$ are positive functions, $g$ is a sigmoidal function, and $I_i$ models external influence by media. 

Here the state space is ${\cal S}=\R$ and the weight space is ${\cal W}=\R^+$, one easily notices that the dynamics of the weights $w_{ij}$ cannot lead to a negative value if $g$ is a nonnegative function. The arising equation for $\mu_t^N$, ignoring the external influence $I$, is given by 
\begin{align}  \label{eq:boschi}
\partial_t \mu^N_t =& - \frac{1}N  \sum_{i=1}^N  \sum_{j \neq i} \nabla_{s_i} \cdot (  (- s_i + \frac{1}N \sum_{j \neq i} w_{ij} g(s_j)    \mu_t^N ) 
\nonumber \\ & -  \sum_{i=1}^N  \sum_{j \neq i}  \nabla_{w_{ij}} \cdot (  \gamma ( J_0 g(s_i) g(s_j) - w_{ij}) \mu_t^N )  + \frac{\sigma^2}2 \sum_{i=1}^N \Delta_{s_i} \mu_t^N.
\end{align} 
We see that \eqref{eq:boschi} is a special case of \eqref{eq:general}  with diffusion coefficients $Q=\frac{\sigma^2}2$, $R=0$, external force $U_0(s_i) = - s_i$ and interactions
$$  U(s_i,s_j,w_{ij}) =     w_{ij} g(s_j), \qquad V(s_i,s_j,w_{ij}) =\gamma ( J_0 g(s_i) g(s_j) - w_{ij}) .$$
%We see that in the case $\sigma = 0$ the prerequisites for weight concentration are met, we find 
%$$ (w_{ij} - \overline{w}_{ij})(V(s_i,s_j,w_{ij}) - V(s_i,s_j,\overline{w}_{ij})) = - \gamma |w_{ij} - \overline{w}_{ij}|^2, $$
%and
%$$ \int (w_{ij} - \overline{w}_{ij}) V(s_i,s_j,\overline{w}_{ij}) \mu_t^N(dz_N) = 
%\gamma  J_0 \int (w_{ij} - \overline{w}_{ij}) g(s_i) g(s_j) \mu_t^N(dz_N) . $$
%Let $\lambda^N_t$ be defined by
%$$ \int \varphi(s_i,s_j)  \lambda^N_t(ds_i,ds_j) = \gamma  J_0 \int \varphi(s_i,s_j)  w_{ij} \mu_t^N(dz_N) $$
%then
%$$ \int g(s_i) g(s_j)  $$
\end{example}

\subsubsection{Continuous States and Binary Weights}

An important case arising in many models is the one of a continuous state space ${\cal S}$ combined with an unweighted graph, which can be rephrased as a set of binary weights ${\cal W}=\{0,1\}$. A corresponding formulation of the kinetic model is given by
\begin{equation}
\partial_t \mu^N_t + \frac{1}N \sum_{i=1}^N \sum_{j\neq i} \nabla_{s_i} \cdot (U(s_i,s_j,w_{ij}) \mu_t^N) =
\sum_{i=1}^N \sum_{j\neq i}  (V(s_i,s_j,w_{ij}') \mu_t^{N,ij} - V(s_i,s_j,w_{ij} ) \mu_t^N),
\end{equation} 
where $w_{ij}'=1-w_{ij}$ and $\mu_t^{N,ij}$ equals $\mu_t^N$ with argument $w_{ij}$ changed to $w_{ij}'$. 
In this case $D_{s_i} = \nabla_{s_i}$ and 
$$ D_{w_i}^* \varphi = \varphi^{ij} - \varphi, $$
where $\varphi^{ij}$ denotes the evaluation of argument $w_{ij}$ changed to $w_{ij}'$. 

\begin{example}In recent variants of bounded confidence models an averaging process of the form 
\begin{equation}
ds_i = \frac{1}N \sum_{j \neq i} w_{ij} (F(s_j) - s_i)~dt
\end{equation}
is carried out on opinions with a linear or nonlinear function $F: \R^+ \rightarrow \R^+$, e.g. $F(s) = k \tanh(\alpha s)$ as in \cite{baumann}. A new edge between $i$ and $j$ is established with a rate $\frac{1}\tau r(|s_i-s_j|)$, $\tau$ a small relaxation time and $r$ being a decreasing function on $\R^+$. The links are removed quickly, which again corresponds to a rate $\frac{1}\tau$. Thus, we obtain a special case of \eqref{eq:general} with ${\cal S}=\R^+$, ${\cal W}=\{0,1\}$, $U_0=Q=R=0$ and 
$$ U(s_i,s_j,w_{ij}) = w_{ij}(F(s_j) - s_i), \qquad V(s_i,s_j,w_{ij}) = \frac{1}\tau (r( |s_i-s_j|) - w_{ij}). $$
In the limit $\tau \rightarrow 0$ we recover bounded confidence models like the celebrated Hegselmann-Krause model \cite{hegselmann} with weight $w_{ij} =  r( |s_i-s_j|). $
\end{example} 

\subsubsection{Discrete States and Binary Weights}

A variety of models (cf. \cite{benatti,maia,min,pham,pham2,raducha,thurner0,thurner}) is based on discrete states (often binary ${\cal S}=\{-1,1\}$) and binary weights (${\cal W}=\{0,1\}$), the interactions consequently being a switching of states between connected particles ($w_{ij}=1$) and the addition and removal of links (changes of weigths between $0$ and $1$ depending on the states). The corresponding evolution of the $N$-particle and weight distribution is thus described by
\begin{align*}
\partial_t \mu_t^N =&  \frac{1}N \sum_{i=1}^N \sum_{j \neq i} \left(\sum_{s_j^* \in {P(s_i)}} U(s_i^*,s_j^*,w_{ij}) \mu_t^{N,s,ij}-U(s_i,s_j,w_{ij}) \mu_t^N \right) + 
\\ & \sum_{i=1}^N \sum_{j \neq i} (V(s_i,s_j,w_{ij}') \mu_t^{N,w,ij} \mu_t^N-V(s_i,s_j,w_{ij}) \mu_t^N)
\end{align*}
where $P(s_i)$ is the set of pre-collisional states $s_j^*$ such that there exists $s_i^* \in {\cal S}$, whose interaction with $s_j^*$ leads to post-collisional state $s_i$.

\begin{example}
The co-evolving voter model, as considered in \cite{min,raducha,thurner0,thurner} is a generic example of a discrete state and weight model. Originally it is formulated in discrete time steps, where in each step a node is picked at random. Then, a neighbouring node is chosen, whose state $s^*$ is compared with the state $s$ of the original node. If $s$ and $s^*$ differ, the link is removed (i.e. the weight is changed from one to zero) with probability $p$, while with probability $1-p$ the state $s$ is changed to $s'$. In the first case another node is chosen randomly and connected to the first one if it has the same state $s$. As a continuous time analogue (with continuous waiting times between events) we obtain the $N$-particle equation
\begin{align}
\partial_t \mu_t^N =&  \frac{1}N \sum_{i=1}^N p\left( \sum_{j \neq i, w_{ij}=0, s_j \neq s_i} \frac{\sum_{k \neq i,j, w_{ik}=1, s_k = s_i}\mu_t^{N,ijk}}{\sum_{k \neq i,j, w_{ik}=1, s_k = s_i} 1} -
 \sum_{j \neq i, w_{ij}=1, s_j \neq s_i}  \mu_t^N \right) +  \nonumber
\\ & \frac{1}N \sum_{i=1}^N (1-p) \left( \sum_{j \neq i, w_{ij}=1, s_j = s_i} \mu_t^{N,i} - \sum_{j \neq i, w_{ij}=1, s_j \neq s_i} \mu_t^N\right) \\
=& \frac{1-p}N \sum_{i=1}^N \sum_{j \neq i} w_{ij} (\mu_t^{N,i} |s_j -s_i'| - \mu_t^{N } |s_j -s_i |  ) + \nonumber \\&
\frac{ p}N \sum_{i=1}^N \sum_{j \neq i} \left( w_{ij}' |s_j-s_i|\frac{\sum_{k \neq i,j}w_{ik} |s_k - s_i'| \mu_t^{N,ijk} }{\sum_{k \neq i,j } w_{ik} |s_k - s_i'|}  - w_{ij}|s_j-s_i| \mu_t^N\right) \nonumber
\end{align}
Here $\mu_t^{N,i}$ denotes $\mu_t^N$ with state $s_i$ changed to $s_i'$ and $\mu_t^{N,ijk}$ with weights $w_{ij}, w_{ik}$ changed to 
 $w_{ij}', w_{ik}'$.
This fits into the modelling  using
$$ D_{s_i} \varphi = \varphi^i - \varphi, \quad D_{w_{ij}}   \varphi = \varphi^{ij} - \varphi $$
and $ U(s_i,s_,w_{ij}) = (1-p) w_{ij} |s_i-s_j|$, up to the three particle interaction with $k$.

We can also consider a variant where a link to a neighbouring node is changed with probability $pq$ and a link to a new node is established with probability $p(1-q)$, which leads to 
\begin{align}
\partial_t \mu_t^N  
=& \frac{1-p}N \sum_{i=1}^N \sum_{j \neq i} w_{ij} (\mu_t^{N,i} |s_j -s_i'| - \mu_t^{N } |s_j -s_i |  ) + \nonumber \\&
{ pq} \sum_{i=1}^N \sum_{j \neq i} \left( w_{ij}' |s_j-s_i| \mu_t^{N,ij}    - w_{ij}|s_j-s_i| \mu_t^N\right) + \nonumber \\&
{ p(1-q)} \sum_{i=1}^N \sum_{j \neq i} \left( w_{ij}' |s_j-s_i'| \mu_t^{N,ij}    - w_{ij}|s_j-s_i'| \mu_t^N\right) .
\end{align}
Here we exactly obtain a special case of \eqref{eq:general} with the above discrete operators $D_{s_i}$ and $D_{w_{ij}}$ as well as the potentials
\begin{equation}
U(s_i,s_j,w_{ij}) = (1-p) w_{ij}|s_j -s_i |, \quad V(s_i,s_j,w_{ij}) =  p w_{ij}( q |s_j -s_i | + (1-q)|s_j -s_i' |).
\end{equation}
\end{example}

\begin{example} {\bf A Minimal Model. } 
As a basis for further analysis we consider a minimal model with binary states ${\cal S} = \{-1,1\}$ and binary weights ${\cal W}=\{0,1\}$, where interaction between states only appears between nodes that are connected ($w_{ij}=1$). This leads to an equation of the form
\begin{align}
\partial_t \mu_t^N =&  \frac{1}N \sum_{i=1}^N \sum_{j \neq i} w_{ij}( \alpha(s_i',s_j) \mu_t^{N,i}-\alpha(s_i,s_j) \mu_t^N ) + \nonumber
\\ & \sum_{i=1}^N \sum_{j \neq i} \left( \beta(s_i,s_j)   (w_{ij}\mu_t^{N, ij} - w_{ij}' \mu_t^N) + \gamma(s_i,s_j)   (w_{ij}'\mu_t^{N, ij} - w_{ij} \mu_t^N) \right) , \label{eq:minimal}
\end{align}
where $\alpha$ is the rate of changing states, $\beta$ the rate to establish links and $\gamma$ the rate of removing links.
Here we denote by $ \mu_t^{N,i}$ the version of $\mu_t^{N}$ with argument $s_i'=-s_i$ instead of $s_i$ and by $\mu_t^{N, ij} $ the version with argument $w_{ij}'=1-w_{ij}$ instead of $w_{ij}$.
\end{example}
%
%
%\subsection{Semi-discrete Opinion Formation Model}
%
%In\cite{baumann} a semi-discrete model is studied 
%\begin{align*}
%\frac{ds_i}{dt} = - s_i + K \sum_{j\neq i} w_{ij}\tanh(\alpha s_j)
%\end{align*} 
%with $w_{ij}(t) \in \{0,1\}$, coupled to a discrete time process where agent $i$ becomes active with some rate and then changes one of the $w_{ij}$ from zero to one with probability
%$$p_{ij} = \frac{|s_i - s_j|^{-\beta}}{\sum_{k \neq i} |s_i - s_k|^{-\beta}}. $$
%
%%\begin{figure}
%\begin{center}
%\includegraphics[width=0.35\textwidth]{u1.png} \includegraphics[width=0.35\textwidth]{v1.png}
 %
%\caption{Evolution of $u$ (left) and $v$ (right) at time steps $t=1,2,3,4,5$ with the network structured model in 
%full lines and the nonlocal reaction-diffusion model in dash-dotted lines. \label{fig1}}
%\end{center}
%\end{figure}

\section{Moment Hierarchies and Closure Relations}

The classical transition from microscopic to macroscopic models in kinetic theory is based on a hierarchy of moments, also called {\em BBGKY hierarchy} in the standard setting (cf. \cite{cercignani3,golse,spohn}) or {\em Vlasov hierarchy} in the corresponding setting (cf. \cite{golsemouhot,spohnneunzert}). We want to derive a similar approach in the following, where we have to take into account the additional correspondence on the weights however. For the sake of simpler exposition and notation we will restrict ourself to the case of an undirected graph without loops, but an analogous treatment is possible for directed weighted graphs.  Our approach will be to consider a system of moments for $k$ vertices and the corresponding edge weights between them. This means the first moment just depends on $s_1$, the second on $s_1, s_2$ and $w_{12}$, the $k$-th on $k$ states and 
$\frac{k(k-1)}2$ weights.

Let us denote by $z_k = (s_i,w_{ij})_{1 \leq i,j \leq k, i \neq j}$ a vector corresponding to the first $k$ states and the weights between them and by $z^{N,k}$ the full vector of states and weights without $z_k$.  Then we define the corresponding moment
\begin{equation}
\mu^{N:k}_t = \int \mu^N_t(dz^{N,k}),
\end{equation} 
with the obvious modification to a sum in the case of a discrete measure.
Integrating \eqref{eq:liouville} we actually see that the moments $\mu^{N:k}_t$ satisfy a closed system of equations, we have for $k=1,\ldots,N$
\begin{align} 
& \partial_t \mu^{N:k}_t = \frac{1}N  \sum_{i \leq k}  \sum_{j \neq i, j \leq k} D_{s_i}^*   (  U(s_i,s_j,w_{ij}) \mu_t^{N:k} ) 
+ \sum_{i \leq k}  \sum_{j \neq i, j \leq k}  D_{w_{ij}}^*  ( V(s_i,s_j,w_{ij}) \mu_t^{N:k} )   
% \nonumber \\ & \qquad \qquad
  \nonumber \\ & \qquad \qquad  + \frac{N-k}N \sum_{i \leq k} D_{s_i}^* ( \int U(s_i,s_{k+1},w_{ik+1}) \mu_t^{N:k+1} (dz^{k+1,k}  )    
	+ \sum_{i \leq k}  D_{s_i}^* (U_0(s_i) \mu_t^{N:k}) \nonumber \\ & \qquad \qquad - \sum_{i\leq k}   D_{s_i}^*  D_{s_i} (Q(s_i) \mu_t^{N:k}) - \sum_{i\leq k} \sum_{j \neq i, j \leq k} D_{w_{ij}}^*D_{w_{ij}}(R(s_i,s_j,w_{ij}) \mu_t^{N:k}),
\end{align} 
where the integration with the measure $ \mu_t^{N:k+1} $ is on the variables $z^{k+1,k}=(s_{k+1},w_{1k+1},\ldots,w_{kk
+1}).$

\subsection{Infinite Hierarchy}

In the limit $N\rightarrow \infty$ we formally obtain the infinite-system
\begin{align} 
& \partial_t \mu^{\infty:k}_t  =   \sum_{i \leq k} D_{s_i}^* ( \int U(s_i,s_{k+1},w_{ik+1}) \mu_t^{\infty:k+1}(dz^{k+1,k}) )    
  \nonumber \\ & \qquad \qquad  + \sum_{i \leq k}  \sum_{j \neq i, j \leq k} D_{w_{ij}}^* ( V(s_i,s_j,w_{ij}) \mu_t^{\infty:k} )   
	+ \sum_{i \leq k}  D_{s_i}^* (U_0(s_i) \mu^{\infty:k}_t) \nonumber \\ & \qquad \qquad - \sum_{i\leq k}   D_{s_i}^*  D_{s_i} (Q(s_i)  \mu^{\infty:k}_t) - \sum_{i\leq k} \sum_{j \neq i, j \leq k} D_{w_{ij}}^*D_{w_{ij}}(R(s_i,s_j,w_{ij})  \mu^{\infty:k}_t)
\end{align} 
which has a similar structure as the BBGKY/Vlasov-hierarchy in kinetic theory. The closedness of the infinite system confirms our choice of marginals in the space of states and weights. 

Let us mention that the first marginal is just the particle density in state space and satisfies 
\begin{equation}
 \partial_t \mu_t^{\infty:1}  =      D_{s_1}^*   \left[ (\int U(s_1,s_2,w_{12}) \mu_t^{\infty:2}(ds_2,dw_{12}) + U_0(s_1)\mu_t^{\infty:1}\right]   - D_{s_1}^*  D_{s_1} (Q(s_1) \mu_t^{\infty:1})
\end{equation} 
It is apparent that there is no simple closure or propagation of chaos in terms of $\mu_t^{\infty:1}$ as long as there is any nontrivial dependence of $U$ and $\mu_t^{\infty:2}$ on the weight $w_{12}$.  As we shall see below, at least in the paradigmatic Vlasov-type model there is some kind of propagation of chaos for the system in terms of the states only, if the measures $\mu_t^{\infty:k}$ exhibit weight concentration phenomena, i.e., concentrate at $w_{ij} = W(s_i,s_j,t)$ for some function $W$.
For the description of the system and its nontrivial network structure the second marginal  $\mu_t^{\infty:2}$ appears to be the more relevant quantity anyway. It satisfies
\begin{align} \label{eq:pairunclosed}
\partial_t \mu^{\infty:2}_t  =&   \sum_{i \leq 2} D_{s_i}^* \left[\int U(s_i,s_{3},w_{i3}) \mu_t^{\infty:3} (ds_3 dw_{13} dw_{23}) + 
U_0(s_i)\mu_t^{\infty:2}- D_{s_i} (Q(s_i)  \mu^{\infty:2}_t) \right]     \nonumber \\ &  +   D_{w_{12}}^* \left[ V(s_1,s_2,w_{12}) \mu_t^{\infty:2}  -    D_{w_{12}}(R(s_1,s_2,w_{12})  \mu^{\infty:2}_t) \right]
\end{align}
We shall below discuss closure relations that approximate the solutions solely in terms of $\mu^{\infty:2}_t$.

\subsection{Pair Closures}

In the following we discuss different options for obtaining a closure in \eqref{eq:pairunclosed}.
A standard closure relation used in statical mechanics is the so-called Kirkwood closure (cf. \cite{kirkwood,singer}), which approximates the triplet distribution by 
pair and single particle distributions, more precisely it assumes a factorization of the triplet correlation into pair correlations. 
The analogous form for our setting is given by
\begin{equation} \label{eq:kirkwood}
\mu_t^{ 3}(dz_3) = \frac{\mu_t^{ 2}(ds_1 ds_2 dw_{12})}{\mu_t^1(ds_1)\mu_t^1(ds_2)}  \frac{\mu_t^{ 2}(ds_1 ds_3 dw_{13})}{\mu_t^1(ds_1)\mu_t^1(ds_3)}    \frac{\mu_t^{ 2}(ds_2 ds_3 dw_{23})}{\mu_t^1(ds_2)\mu_t^1(ds_3)}  \mu_t^1(ds_1)\mu_t^1(ds_2)\mu_t^1(ds_3)
\end{equation}
where the quotients denote the respective Radon-Nikodym derivatives. The Kirkwood closure might yield an overly complicated system for the pair distribution, which we see by examining the relevant terms, namely the integrals of $U$ in \eqref{eq:pairunclosed}. For $i=1$ we have
$$ \int U(s_1,s_{3},w_{13}) \mu_t^{\infty:3} (ds_3 dw_{13} dw_{23}) = \int U(s_1,s_{3},w_{13})  \eta_t^2  \mu_t^{ 2}(ds_1 ds_2 dw_{12}) \frac{\mu_t^{ 2}(ds_1 ds_3 dw_{13})}{\mu_t^1(ds_1) }   
$$
where $\eta_t^2$ is the projection of the Radon-Nikodym derivative of $\mu_t^2$ to the first two variables, i.e. for each set $A \subset {\cal S}^2$,
$$ \int_A \eta_t^2 \mu_t^1(ds_2)\mu_t^1(ds_3) = \int_A \int_{{\cal W}}  {\mu_t^{ 2}(ds_2 ds_3 dw_{23})}  . $$
Note that $\frac{\mu_t^{ 2}(ds_1 ds_3 dw_{13})}{\mu_t^1(ds_1) }$ is the conditional distribution of $s_3$ and $w_{13}$ given $s_1$.

As argued for a similar class of particle systems (without the network weights) in \cite{berlyand1,berlyand2}, $\eta_t^2$ seems to be unnecessary for a suitable approximation of the integrals, hence their closure simplifies to 
\begin{equation} \label{eq:berlyand}
\int U(s_i,s_{3},w_{i3}) \mu_t^{\infty:3} (ds_3 dw_{13} dw_{23}) = \int U(s_i,s_{3},w_{i3})  \mu_t^{ 2}(ds_1 ds_2 dw_{12}) \frac{\mu_t^{ 2}(ds_i ds_3 dw_{i3})}{\mu_t^1(ds_1) }   .
\end{equation}
In order to state the arising equation it is more convenient to use the weak formulation, for which we obtain
\begin{align}  
\frac{d}{dt} \int \varphi(z_2) \mu^{ 2}_t(dz_2)  =&   
\sum_{i \leq 2}  \int D_{s_i} \varphi(z_2)   (U(s_i,s_{3},w_{i3}) \gamma_t(s_i;ds_3,dw_{i3})    + 
U_0(s_i)\mu_t^{2}(dz_2 )) -  \nonumber \\ &
\sum_{i \leq 2}  \int D_{s_i} \varphi(z_2)  D_{s_i} (Q(s_i)  \mu^{2}_t ) (dz_2 )
  +  \nonumber \\ & \int D_{w_{12}} \varphi(z_2)   V(s_1,s_2,w_{12}) (\mu_t^{ 2}(dz_2 )  -    D_{w_{12}}(R(s_1,s_2,w_{12})  \mu^{2}_t)(dz_2 ) ) 
\end{align} 
with the conditional distribution
$$ \gamma_t(s_i;ds_3,dw_{i3}) = \frac{\mu_t^{ 2}}{\mu_t^{ 1}}(ds_3,dw_{i3}).$$

\section{Minimal Model}

In the following we further study the minimal model \eqref{eq:minimal} with the above closures. With the short-hand notations
\begin{align}
f_{\pm,\pm} &= \mu_t^2(\pm 1, \pm 1, 1), \quad g_{\pm,\pm} = \mu_t^2(\pm 1, \pm 1, 0), \\
 \rho_+ &= f_{++} + g_{++} + f_{+-} + g_{+-} \\
 \rho_- &= f_{--} + g_{--} + f_{+-} + g_{+-}
\end{align} 
noticing $f_{+-} = f_{-+}$ due to symmetry, we arrive at the following models: In the case of the closure based on the conditional distribution \eqref{eq:berlyand} we have
\begin{align}
\partial_t f_{++} =& \alpha_{-+} \frac{f_{+-}^2}{\rho_-} - \alpha_{+-} \frac{f_{++} f_{+-}}{\rho_+} + \beta_{++} g_{++} - \gamma_{++} f_{++}
\label{eq:minimal1} \\
\partial_t g_{++} =& \alpha_{-+} \frac{g_{+-}f_{+-}}{\rho_-} - \alpha_{+-} \frac{g_{++} f_{+-}}{\rho_+} - \beta_{++} g_{++} + \gamma_{++} f_{++} \\
\partial_t f_{--} =&  \alpha_{+-} \frac{f_{+-}^2}{\rho_+} - \alpha_{-+} \frac{f_{--} f_{+-}}{\rho_-} + \beta_{--} g_{--} - \gamma_{--} f_{--} \\
\partial_t g_{--} =& \alpha_{+-} \frac{g_{+-}f_{+-}}{\rho_+} - \alpha_{-+} \frac{g_{--} f_{+-}}{\rho_-} - \beta_{--} g_{--} + \gamma_{--} f_{--} \\
\partial_t f_{+-} =&  - \alpha_{-+} \frac{f_{+-}^2}{2\rho_-} + \alpha_{+-} \frac{f_{++} f_{+-}}{2\rho_+} - \alpha_{+-} \frac{f_{+-}^2}{2\rho_+} + \alpha_{-+} \frac{f_{--} f_{+-}}{2\rho_-} \nonumber \\
& +\beta_{+-} g_{+-} - \gamma_{+-} f_{+-} \\
\partial_t g_{+-} =&  - \alpha_{-+} \frac{g_{+-}f_{+-}}{2\rho_-} + \alpha_{+-} \frac{g_{++} f_{+-}}{2\rho_+}  - \alpha_{+-} \frac{g_{+-}f_{+-}}{2\rho_+} + \alpha_{-+} \frac{g_{--} f_{+-}}{2\rho_-} \nonumber \\ & - \beta_{+-} g_{+-} + \gamma_{+-} f_{+-} \label{eq:minimal6} 
\end{align}
We can also introduce the weight-averaged densities $h_{\pm\pm} = f_{\pm\pm} + g_{\pm\pm}$, which satisfy the equations
\begin{align}
\partial_t h_{++} =& \alpha_{-+} \frac{h_{+-}f_{+-} }{\rho_-} - \alpha_{+-} \frac{h_{++} f_{+-}}{\rho_+} \\
\partial_t h_{--} =& \alpha_{+-} \frac{h_{+-}f_{+-}}{\rho_+} - \alpha_{-+} \frac{h_{--} f_{+-}}{\rho_-} \\
\partial_t h_{+-} =& - \frac{1}2 (\partial_t h_{++} + \partial_t h_{--}),
\end{align}
the latter being equal to the conservation property $\partial_t (\rho_- + \rho_+) = 0$.

For the Kirkwood closure we find instead
\begin{align}
\partial_t f_{++} =& \alpha_{-+} \frac{f_{+-}^2}{\rho_-} \frac{h_{++}}{\rho_+^2} - \alpha_{+-} \frac{f_{++} f_{+-}}{\rho_+} \frac{h_{+-}}{\rho_+ \rho_-}+ \beta_{++} g_{++} - \gamma_{++} f_{++}
\label{eq:minimalk1} \\
\partial_t g_{++} =& \alpha_{-+} \frac{g_{+-}f_{+-}}{\rho_-} \frac{h_{++}}{\rho_+^2} - \alpha_{+-} \frac{g_{++} f_{+-}}{\rho_+} \frac{h_{+-}}{\rho_+ \rho_-}- \beta_{++} g_{++} + \gamma_{++} f_{++} \\
\partial_t f_{--} =&  \alpha_{+-} \frac{f_{+-}^2}{\rho_+} \frac{h_{--}}{\rho_-^2}- \alpha_{-+} \frac{f_{--} f_{+-}}{\rho_-} \frac{h_{+-}}{\rho_+ \rho_-}+ \beta_{--} g_{--} - \gamma_{--} f_{--} \\
\partial_t g_{--} =& \alpha_{+-} \frac{g_{+-}f_{+-}}{\rho_+} \frac{h_{--}}{\rho_-^2} - \alpha_{-+} \frac{g_{--} f_{+-}}{\rho_-} \frac{h_{+-}}{\rho_+ \rho_-}- \beta_{--} g_{--} + \gamma_{--} f_{--} \\
\partial_t f_{+-} =&  - \alpha_{-+} \frac{f_{+-}^2}{2\rho_-}  \frac{h_{++}}{\rho_+^2} + \alpha_{+-} \frac{f_{++} f_{+-}}{2\rho_+} \frac{h_{+-}}{\rho_+ \rho_-} - \alpha_{+-} \frac{f_{+-}^2}{2\rho_+} \frac{h_{--}}{\rho_-^2} \nonumber \\
&+ \alpha_{-+} \frac{f_{--} f_{+-}}{2\rho_-} \frac{h_{+-}}{\rho_+ \rho_-} +\beta_{+-} g_{+-} - \gamma_{+-} f_{+-} \\
\partial_t g_{+-} =&  - \alpha_{-+} \frac{g_{+-}f_{+-}}{2\rho_-}  \frac{h_{++}}{\rho_+^2}+ \alpha_{+-} \frac{g_{++} f_{+-}}{2\rho_+} \frac{h_{+-}}{\rho_+ \rho_-}  - \alpha_{+-} \frac{g_{+-}f_{+-}}{2\rho_+} \frac{h_{--}}{\rho_-^2}\nonumber \\ &+ \alpha_{-+} \frac{g_{--} f_{+-}}{2\rho_-} \frac{h_{+-}}{\rho_+ \rho_-}  - \beta_{+-} g_{+-} + \gamma_{+-} f_{+-} \label{eq:minimalk6} 
\end{align}
The equations for  the weight-averaged densities are given by
\begin{align}
\partial_t h_{++} =& \alpha_{-+} \frac{h_{+-}f_{+-} }{\rho_-} \frac{h_{++}}{\rho_+^2}  - \alpha_{+-} \frac{h_{++} f_{+-}}{\rho_+} \frac{h_{+-}}{\rho_+ \rho_-} \\
\partial_t h_{--} =& \alpha_{+-} \frac{h_{+-}f_{+-}}{\rho_+} \frac{h_{--}}{\rho_-^2} - \alpha_{-+} \frac{h_{--} f_{+-}}{\rho_-} \frac{h_{+-}}{\rho_+ \rho_-}\\
\partial_t h_{+-} =& - \frac{1}2 (\partial_t h_{++} + \partial_t h_{--}).
\end{align}
 
\subsection{Transient Solutions}

In order to provide a first analysis of solutions to \eqref{eq:minimal1}-\eqref{eq:minimal6} a key observation is to consider the 
evolution of the single particle densities $\rho_+$ and $\rho_-$ respectively. We find
\begin{equation} \label{eq:rhoplusevolution}
\partial \rho_+ = - \partial_t \rho_- = \left( \frac{\alpha_{-+}}2 - \frac{\alpha_{+-}}2\right) f_{+-},
\end{equation} 
in the case of \eqref{eq:minimal1}-\eqref{eq:minimal6}, and
\begin{equation} \label{eq:rhoplusevolutionk}
\partial \rho_+ = - \partial_t \rho_- = \left( \frac{\alpha_{-+}}2 - \frac{\alpha_{+-}}2\right) \frac{\rho_- h_{++}+ \rho_+ h_{--}}{\rho_+ \rho_-} \frac{h_{+-}}{\rho_+ \rho_-}f_{+-},
\end{equation} 
in the case of \eqref{eq:minimalk1}-\eqref{eq:minimalk6},
which implies first of all that $\rho_+$ is conserved if $\alpha_{-+} = \alpha_{-+}$. We also find the natural properties that $\rho_+$ increases if $\alpha_{-+} > \alpha_{-+}$, i.e. if there is a stronger bias towards the positive state, and decreases if $\alpha_{-+} < \alpha_{-+}$. In these cases we see immediately that concentration (consensus in the language of opinion formation) is reached if $f_{+-}$ is bounded away from zero. With a standard proof (see Appendix) we obtain the following result:

\begin{thm} \label{minimalexistencethm}
For each nonnegative starting value with $\rho_+(0) \in (0,1)$, $\rho_-(0) = 1 - \rho_+(0)$, there exists a time $T_* > 0$ such that there exists a unique solution 
$$ (f_{++},g_{++},f_{--},g_{--},f_{+-},g_{+-}) \in C^\infty([0,T_*))^6 $$
of \eqref{eq:minimal1}-\eqref{eq:minimal6} as well as of \eqref{eq:minimalk1}-\eqref{eq:minimalk6}. Moreover, either $T_* = \infty$ or $\rho_+(t) \rho_-(t) \rightarrow 0$
as $t \uparrow T_*$. If $\alpha_{-+} = \alpha_{-+}$ then $T_* = \infty$.
\end{thm}
%\begin{proof}
%\end{proof}

If $T_*$ is finite we thus see that either the limit of $\rho_+$ or $\rho_-$ vanishes and since $f_{+-} \leq \rho_+$ and $f_{+-} \leq \rho_-$ we obtain $f_{+-}(t) \rightarrow 0$ as well. This is not surprising, since such as state means consensus at one of the two states and hence there are no agents of another state to connect to. A possibly more interesting situation can happen if $T_* = \infty$, then we might have polarization, i.e. $f_{+-} = 0$ with $\rho_+$ and $\rho_-$ both being positive. This means there are agents of either state but no connection between them. A direct insight can be obtained by simply integrating \eqref{eq:rhoplusevolution} in time, which yields
$$ 1 \geq \left( \frac{\alpha_{-+}}2 - \frac{\alpha_{+-}}2\right) \int_0^\infty f_{+-}(t)~dt. $$ 
If the coefficient on the right-hand side is vanishing, $f_{+-}$ is integrable, which implies the following result:

\begin{cor}
Let $\alpha_{-+} \neq \alpha_{-+}$ and let $ (f_{++},g_{++},f_{--},g_{--},f_{+-},g_{+-}) \in C^\infty([0,T_*))^6 $ be the unique solution of \eqref{eq:minimal1}-\eqref{eq:minimal6}. If $T_* = \infty$, then $f_{+-}(t) \rightarrow 0$ as $t \rightarrow \infty$.
\end{cor} 

As a consequence of this result we may expect to find polarized stationary solutions of the minimal model with closure based on the conditional distribution \eqref{eq:minimal1}-\eqref{eq:minimal6}, which we further investigate in an asymptotic parameter regime below.
A similar analysis for the minimal model with Kirkwood closure \eqref{eq:minimalk1}-\eqref{eq:minimalk6} is less obvious due to the additional terms in \eqref{eq:rhoplusevolutionk}. With $\rho_+ \leq 1$, $\rho_- \leq 1$, and $h_{+-} \geq f_{+-}$ we can at least infer the 
integrability of $ (h_{++} + h_{--}) f_{+-}^2, $ which yields $f_{+-} \rightarrow 0$ or $h_{++} + h_{--} \rightarrow 0$. The latter can only happen if $\rho_+ - \rho_- = h_{+-} - h_{+-} \rightarrow 0$. This kind of limiting solution with $h_{++}=h_{--}=0$ and $h_{+-}=\frac{1}2$ actually appears as a deficiency of the Kirkwood closure to describe the many-particle limit. Such a solution can appear in the case of two particles, with one in the $+$ and one in the $-$ state. As soon as there are three or more particles, we will find at least two of them in the same state, which implies $h_{++} + h_{--} > 0$. 

\subsection{Stationary Solutions and Polarization}

In order to understand possible segregation phenomena we study the stationary solutions of \eqref{eq:minimal1}-\eqref{eq:minimal6}, focusing in particular on the case when two nodes are mainly connected if they have the same state and rather disconnected if the have opposite states. This means that $\beta_{+-}$ is small and $\gamma_{+-}$ is rather large. 

We first study the stationary equations for $h_{++}$ and $h_{--}$ in the case of the conditional distribution based closure, which yields $f_{+-}=0$ or
$$ h_{++} = \frac{\alpha_{-+} \rho_+ h_{+-}}{\alpha_{+-} \rho_-}, \quad  h_{--} = \frac{\alpha_{+-} \rho_- h_{+-}}{\alpha_{-+} \rho_+}.$$
Together with 
$$ 1 = \rho_+ + \rho_- = h_{++} + h_{--} + 2 h_{+-} $$
we arrive at 
\begin{align*} &h_{++} = \frac{\alpha_{-+}^2 \rho_+^2}{(\alpha_{-+} \rho_+ + \alpha_{+-} (1-\rho_+))^2}, \quad  h_{--}=\frac{ \alpha_{+-}^2  (1- \rho_+)^2}{(\alpha_{-+} \rho_+ + \alpha_{+-} (1-\rho_+))^2},   \\
&h_{+-}=\frac{\alpha_{-+}\alpha_{+-} \rho_+ (1- \rho_+)}{(\alpha_{-+} \rho_+ + \alpha_{+-} (1-\rho_+))^2}. 
\end{align*}
Let us mention that in the case $\alpha_{+-} = \alpha_{-+} = \alpha$, they considerably simplify to
$$   h_{++} =   \rho_+^2 , \quad  h_{--}= \rho_-^2 = (1- \rho_+)^2 , \quad  h_{+-}=  \rho_+ \rho_- \rho_+ (1- \rho_+) , $$
which corresponds to  mixed states and weights. Thus, in order to understand polarization we will focus on the case $f_{+-} =0$ first.

In the case of the Kirkwood closure solutions with $f_{+-}>0$ (and thus  $h_{+-}>0$) are obtained only for $h_{++}=h_{--}=0$, which are again not the ones relevant for the many particle limit.

\subsubsection{No link creation between opposite states}

In order to understand the segregation phenomenon in the model, let us first consider the most extreme case of $\beta_{+-} = 0$, hence there are no links created between agents of opposite states $+1$ and $-1$. In this case we trivially find stationary solutions with $f_{+-} =0$:

\begin{prop} \label{minimalstatprop}
Let $\beta_{+-}=0$, $\beta_{++}+\gamma_{++} >0$, and $\beta_{--}+\gamma_{--} >0$. Then there exists an infinite number of stationary solution of \eqref{eq:minimal1}-\eqref{eq:minimal6} as well as of \eqref{eq:minimalk1}-\eqref{eq:minimalk6} with $f_{+-}=0$, $\rho_+ \in [0,1]$ arbitrary, $g_{+-} \in [0,\min\{\rho_+,1-\rho_+\}]$ arbitrary,
and 
\begin{align*}
f_{++} &= \frac{\beta_{++}}{\beta_{++}+\gamma_{++}}(\rho_+ - g_{+-}), && 
g_{++} = \frac{\gamma_{++}}{\beta_{++}+\gamma_{++}} (\rho_+ - g_{+-}),\\
f_{--} &= \frac{\beta_{--}}{\beta_{--}+\gamma_{--}} (1-\rho_+ - g_{+-}), && 
g_{--} = \frac{\gamma_{--}}{\beta_{--}+\gamma_{--}} (1-\rho_+ - g_{+-}).
%g_{+-} &=\frac{\alpha_{-+}\alpha_{+-} \rho_+ (1- \rho_+)}{(\alpha_{-+} \rho_+ + \alpha_{+-} (1-\rho_+))^2}.
 \end{align*}
\end{prop}

The above type of stationary solutions yields a polarization, i.e. unless $\rho_ + = 1$ or $\rho_+=0$ the network consists of two parts with states $+1$ and $-1$, and the corresponding nodes with different states are never connected by an edge. 
%\begin{thm}
%Let $\beta_{+-} >0$, $\beta_{++}+\gamma_{++} > 0$, $\beta_{--}+\gamma_{--} > 0$, and
From the linearized equations of \eqref{eq:minimal1}-\eqref{eq:minimal6} around the stationary state (see Appendix \ref{appendixminimal}) we expect the polarization to be stable if
\begin{equation}
\gamma_{+-} >  \frac{\beta_{++}}{\beta_{++}+\gamma_{++}} \frac{ \alpha_{+-} \alpha_{-+}^2 \rho_+}{2 (\alpha_{-+} \rho_+ + \alpha_{+-} (1-\rho_+))^2}  +  \frac{\beta_{--}}{\beta_{--}+\gamma_{--}} \frac{\alpha_{-+} \alpha_{+-}^2 (1-\rho_+)^2}{2(\alpha_{-+} \rho_+ + \alpha_{+-} (1-\rho_+))^2}. \label{minimallinstabcondition}
\end{equation} 
The linearized problem has a threefold zero eigenvalue however, which makes a rigorous analysis based on linear stability difficult. However, we can provide a nonlinear stability result under a slightly stronger condition.
An analogous result with slightly different condition holds for the linearization of \eqref{eq:minimalk1}-\eqref{eq:minimalk6}.
%Then the manifold of stationary solutions from Proposition \ref{minimalstatprop}, in particular $f_{+-} =0$, is linearly stable.
%\end{thm}
%\begin{proof}
%We denote the variables for the linearized problem by $\tilde f_{**}$ and $\tilde g_{**}$, respectively. The linearized problem around a stationary state with $f_{+-} = 0$ is given by
%
%We start with the last two equations, from which we directly see that \eqref{minimallinstabcondition} implies that $\tilde f_{+-}$ decays exponentially.
 %and subsequently that $\tilde g_{+-}$ decays exponentially as well. As a next step we can examine the $2 \times 2$ subsystems for $\tilde f_{++}, \tilde g_{++}$ and $\tilde f_{--}, \tilde g_{--}$, respectively. Taking the difference, we see that 
%$\tilde f_{++} - \tilde g_{++}$ and $\tilde f_{--} - \tilde g_{--}$ decay exponentially. Finally taking the sum, we see that 
%$$ \partial_t (\tilde f_{++} + \tilde g_{++}) = \partial_t (\tilde f_{--} + \tilde g_{--}) = 0, $$
%which finally the implies linear stability and conservation property.
%\end{proof}
%
%Note that the conservation of perturbations in $h_{++}$ and $h_{--}$ is related to changes in $\rho_+ = h_{++} + h_{+-}$ and $\rho_- = h_{--} + h_{+-}$, which of course can lead to other stationary solutions in the infinite family. 
We see that a sufficient condition for \eqref{minimallinstabcondition} is $ 2 \gamma_{+-} > \alpha_{+-} + \alpha_{-+}$, which is even sufficient for nonlinear stability of a polarized state:
\begin{thm} \label{thm:nonlinstab}
Let $ 2 \gamma_{+-} > \alpha_{+-} + \alpha_{-+}$ and let $(f_{++},g_{++},f_{--},g_{--},f_{+-},g_{+-})$ be a nonnegative solution of \eqref{eq:minimal1}-\eqref{eq:minimal6}. Then 
$$ f_{+-}(t) \leq e^{-(\gamma_{+-}- \frac{\alpha_{+-} + \alpha_{-+}}2)t} f_{+-}(0), $$ 
\end{thm}
in particular $f_{+-}(t) \rightarrow 0 $ as $t \rightarrow \infty$.
\begin{proof}
Using the nonnegativity of the solution we have $f_{++} \leq \rho_+$ and $f_{--} \leq \rho_-$ and thus
\begin{align*} \partial_t f_{+-} =&  - \alpha_{-+} \frac{f_{+-}^2}{2\rho_-}  - \alpha_{+-} \frac{f_{+-}^2}{2\rho_+} + \left( \alpha_{+-} \frac{f_{++} }{2\rho_+} + \alpha_{-+} \frac{f_{--} }{2\rho_-}  - \gamma_{+-} \right) f_{+-} \\ 
\leq& \left( \frac{\alpha_{+-} }{2 } +  \frac{\alpha_{-+} }{2 }  - \gamma_{+-} \right) f_{+-} ,
\end{align*}
which implies the assertion due to Gronwall's Lemma.
\end{proof}

An analogous result can be obtained for the Kirkwood closure, with a slightly stronger condition:
\begin{thm}
Let $ \gamma_{+-} > \alpha_{+-} + \alpha_{-+}$ and let $(f_{++},g_{++},f_{--},g_{--},f_{+-},g_{+-})$ be a nonnegative solution of \eqref{eq:minimalk1}-\eqref{eq:minimalk6}. Then 
$$ f_{+-}(t) \leq e^{-(\gamma_{+-}- {\alpha_{+-} + \alpha_{-+}})t} f_{+-}(0), $$ 
\end{thm}
in particular $f_{+-}(t) \rightarrow 0 $ as $t \rightarrow \infty$.
\begin{proof}
The proof is analogous to the one of Theorem \ref{thm:nonlinstab} using the following estimate for the additional term
$$ \frac{h_{+-}}{\rho_+ \rho_-} \leq \min\{ \frac{1}{\rho_+}, \frac{1}{\rho_+} \} \leq 2. $$
\end{proof}

\subsubsection{Rare link creation between opposite states}

Having understood the above extreme case of having no link creation between nodes of opposite states, we can also extend to the case of rare link creation, i.e. $\beta_{+-}$ small. For simple notation we use the notation $\epsilon = \beta_{+-}$ and perform an asymptotic analysis around $\epsilon = 0$ via the implicit function theorem. 

\begin{prop} \label{smallepsilonprop}
Let all parameters $\alpha_{**}, \beta_{**}, \gamma_{**}$ be positive and let
$$ 2 \gamma_{+-} > \alpha_{+-} + \alpha_{-+}.$$ Then there exists $\epsilon_0 > 0$ such that for all $\gamma_{+-}=\epsilon \in [0,\epsilon_0)$  there exists an infinite number of stationary solution of \eqref{eq:minimal1}-\eqref{eq:minimal6} with  $\rho_+ \in [0,1]$ and $f_{+-} = {\cal O}(\epsilon)$.  %The corresponding nonlinear stationary solutions are linearly stable. 
\end{prop}
%\begin{proof}
%Linear stability follows from continuous dependence of the derivative around the stationary solution. 
%\end{proof} 

Let us remark that the condition $2 \gamma_{+-} > \alpha_{+-} + \alpha_{-+}$ is not optimal in Proposition \ref{smallepsilonprop}, it was just used to obtain a condition that holds independent of $\rho_+$. If we are interested in the behaviour at a specific value of $\rho_+$ only it could indeed be replaced by \eqref{minimallinstabcondition}. With an analogous proof we can give a statement for the Kirkwood closure:

\begin{prop}  
Let all parameters $\alpha_{**}, \beta_{**}, \gamma_{**}$ be positive and let
$$   \gamma_{+-} > \alpha_{+-} + \alpha_{-+}.$$ Then there exists $\epsilon_0 > 0$ such that for all $\gamma_{+-}=\epsilon \in [0,\epsilon_0)$  there exists an infinite number of stationary solution of \eqref{eq:minimalk1}-\eqref{eq:minimalk6} with  $\rho_+ \in [0,1]$ and $f_{+-} = {\cal O}(\epsilon)$.  %The corresponding nonlinear stationary solutions are linearly stable. 
\end{prop}

\section{Vlasov-type model}

In the following we further investigate the paradigmatic model based on the Vlasov-type dynamics, corresponding to the microscopic model \eqref{eq:microscopic1}, \eqref{eq:microscopic2} respectively the Liouville equatione \eqref{eq:liouville} and the corresponding infinite hierarchy, given in weak formulation 
\begin{align} 
& \int_0^T \int \partial_t \varphi_k(z_k,t) \mu^{\infty:k}_t (dz_k)   +   \sum_{i \leq k} 
\int_0^T \int \nabla_{s_i} \varphi_k(z_k,t)\cdot U(s_i,s_{k+1},w_{ik+1}) \mu_t^{\infty:k+1}(dz_{k+1} )    
  \nonumber \\ & \qquad \qquad  + \sum_{i \leq k}  \sum_{j \neq i, j \leq k} \int_0^T \int  \nabla_{w_{ij}} 
	\varphi_k(z_k,t) \cdot  V(s_i,s_j,w_{ij}) \mu_t^{\infty:k}(dz_k )  = 0 ,  \label{eq:hierarchyweak}
\end{align} 
with continuously differentiable test functions $\varphi_k$ depending on $z_k$. We start with a case where some kind of propagation of chaos exists, namely if the weights are fully concentrated.

\subsection{Mean-Field Models for Weight Concentration}

As mentioned above there is no propagation of chaos respectively a simple mean-field solution of the infinite hierarchy in general, but we can find such if there is concentration of weights. The latter is not surprising, since once we find weight concentration, the resulting hierarchy is effectively rewritten in terms of the states $s_i$ only. Such states can be found for the paradigmatic model only and are related to special solutions of the form 
$w_{ij}=W(s_i,s_j(t),t)$ of \eqref{eq:microscopic1}, \eqref{eq:microscopic2}. Essentially we look for solutions of
%We further see that there are special solutions for the weights in the above system in the form 
%$$ w_{ij}(t) = W(s_i(t),s_j(t),t), $$
%with $W$ being the solution of 
$$ \frac{d}{dt}W(s_i(t),s_j(t),t)  = V(s_i,s_j,W(s_i(t),s_j(t),t)), $$
which by the chain rule is converted to a transport equation for $W$.

In order to construct weight concentrated mean-field solutions, we look for measures of the form
$$ \mu^{\infty:k}_t(dz_k) = \lambda^k_t(ds_1,\ldots,ds_k) \prod_{i\leq k}\prod_{j \neq i,j \leq k} \delta_{W(s_i,s_j,t)}(dw_{ij}), $$
with some function $W:{\cal S} \times {\cal S} \times (0,T) \rightarrow {\cal W}$ encoding the weights between states $s_i$ and $s_j$. We find
$$ \int_0^T \int  \varphi_k(z_k,t) \mu^{\infty:k}_t(dz_k) = \int_0^T \int \psi_k(s_1,\ldots,s_k,t)\lambda^k_t(ds_1,\ldots,ds_k),$$
where we introduce the short-hand notation
$$ \psi_k(s_1,\ldots,s_k,t) = \varphi_k(s_1,\ldots,s_k,W(s_1,s_2,t),\ldots,W(s_{k-1},s_k,t),t). $$
We notice that 
$$ \partial_t \psi_k = \partial_t \varphi_k + \sum_i \sum_{j \neq i} \partial_{w_{ij}} \varphi_k \cdot \partial_t W(s_i,s_j,t)   $$
%with $W_{ij} = W(s_i,s_j,t)$,
and 
$$ \nabla_{s_i} \psi_k = \nabla_{s_i} \varphi_k + \sum_{j \neq i} \partial_{w_{ij}} \varphi_k \cdot \nabla_{s_1} W(s_i,s_j,t) + \sum_{j \neq i} \partial_{w_{ji}} \varphi_k \cdot \nabla_{s_2} W(s_j,s_i,t) . $$
Inserting these relations for $\partial_t \varphi_k$ and $\nabla_{s_i} \varphi_k$ into \eqref{eq:hierarchyweak} we obtain
\begin{align} 
 & \int_0^T \int \partial_t \psi_k(z_k,t) \lambda^{k}_t(ds_1,\ldots,ds_k)  \nonumber\\ &\qquad  +   \sum_{i \leq k} 
\int_0^T \int \nabla_{s_i} \psi_k(z_k,t)\cdot U(s_i,s_{k+1},W(s_i,s_{k+1},t)) \lambda_t^{k+1}(ds_1,\ldots,ds_{k+1} )
     = R ,
\end{align} 
with the remainder term
\begin{align*} 
R_k =&  \sum_{i \leq k} \sum_{j \neq i} \int_0^T \int \partial_{w_{ij}} \varphi_k \cdot \partial_t W(s_i,s_j,t) \lambda^{k}_t(ds_1,\ldots,ds_k) 
+ \\ &\sum_{i \leq k} \sum_{j \neq i}  \int_0^T \int \partial_{w_{ij}} \varphi_k \cdot \nabla_{s_1} W(s_i,s_j,t)  U(s_i,s_{k+1},W(s_i,s_{k+1},t)) \lambda_t^{k+1}(ds_1,\ldots,ds_{k+1}) + \\ &\sum_{i \leq k} \sum_{j \neq i}  \int_0^T \int \partial_{w_{ij}} \varphi_k \cdot \nabla_{s_2} W(s_i,s_j,t)  U(s_j,s_{k+1},W(s_j,s_{k+1},t)) \lambda_t^{k+1}(ds_1,\ldots,ds_{k+1} ) - \\&
\sum_{i \leq k} \sum_{j \neq i}  \int_0^T \int \partial_{w_{ij}} \varphi_k \cdot V(s_i,s_j,W(s_i,s_j,t)) \lambda^{k}_t(ds_1,\ldots,ds_k) .
\end{align*}

Now we see that there is a special solution of the form 
$$ \lambda^{k}_t(ds_1,\ldots,ds_k) = \prod_{i=1}^k  \lambda_t^1(ds_i), $$
with $\lambda_t^1$ solving the mean-field equation
\begin{equation} \label{eq:charweightconcentration}
 \partial_t \lambda_t^1  +     \nabla_{s_1} \cdot \left[ \lambda_t^1  \int U(s_1,s_2,W(s_1,s_2)) \lambda_t^{1}(ds_2) \right]    = 0.
\end{equation} 
and $R_k$ vanishes if $W$ solves the transport equation
\begin{align*}
\partial_t W(s_1,s_2,t) +  \int U(s_1,s_3,W(s_1,s_3)) \lambda_t^{1}(ds_3) \cdot \nabla_{s_1} W(s_1,s_2,t)   &\\   +\int U(s_2,s_3,W(s_1,s_3)) \lambda_t^{1}(ds_3) \cdot \nabla_{s_2} W(s_1,s_2,t)
&= V(s_1,s_2,W(s_1,s_2,t)) .
\end{align*}

Let us mention that we can solve \eqref{eq:charweightconcentration} by the method of characteristics
\begin{align}
\partial_t S(t;s_1) &= \int  U(S(t;s_1),S(t;s'),\hat W(t;s_1,s'))  \lambda^1_0(ds') \label{eq:char1wc} \\
\partial_t \hat W(t;s_1,s_2) &= V(S(t;s_1),S (t;s_2),\hat W(t;s_1,s_2)) \label{eq:char2wc}
\end{align}
with $\hat W(t;s_1,s_2) := W(S(t;s_1);S(t;s_2). $

\begin{prop} \label{weightconcentrationproof}
Let $U$ and $V$ be Lipschitz continuous and let $\lambda^1_0 \in {\cal M}_+(\R^m))$.
Then there exists a unique solution 
$$ S \in C^1(0,T;C(\R^m)), \qquad \hat W \in C^1(0,T;C(\R^{2m}))$$
of the characteristic system \eqref{eq:char1wc}, \eqref{eq:char2wc}. Moreover, $s \mapsto S(t;s)$ is a one-to-one map on $\R^{m}$ for each $t > 0$.
\end{prop}
\begin{proof}
Due to the Picard-Lindel\"of Theorem we conclude existence and uniqueness of solutions of \eqref{eq:char1wc}, \eqref{eq:char2wc}.
The fact that $s \mapsto S(t;s)$ is injective follows from the uniqueness of the ODE system, the surjectivity from the analogous existence result that can be applied to \eqref{eq:char1wc}, \eqref{eq:char2wc} with terminal conditions.
\end{proof}

As a consequence of the analysis of the characteristics we immediately obtain the existence and uniqueness of a solution of \eqref{eq:charweightconcentration} as the push-forward of $\lambda_0^1$ under $S$.
\begin{cor}
Let $U$ and $V$ be Lipschitz continuous and let $\lambda^1_0 \in {\cal M}_+(\R^m))$. Then there exists a unique solution  $\lambda_t^1 \in C(0,T;{\cal M}_+(\R^m))$ of \eqref{eq:charweightconcentration}.
\end{cor}

\subsection{Closure based on the  conditional distribution}

In the following we provide some further mathematical analysis for the more general case, with the closure based on the conditional distribution. Correspondingly we will consider the equivalent system for the single particle distribution $\mu_t^{ 1}$ and the conditional distribution $\gamma_{t}  = \frac{\mu_t^{ 2}}{\mu_t^{ 1}}$ instead, which is given by
\begin{align} 
\partial_t  \mu^{ 1}_t   +   
  \nabla_{s_1}  \cdot \left(\int   U(s_1,s',w')  \gamma_t(s_1;ds',dw') \mu^1_t \right) &=  0 \label{eq:conddist1} \\
	\partial_t  \gamma_t  + \int   U(s_1,s',w')  \gamma_t(s_1;ds',dw')  \cdot \nabla_{s_1} \gamma_t & \nonumber \\
	+ \nabla_{s_2} \cdot (\int   U(s_2,s',w')  \gamma_t(s_2;ds',dw')  \gamma_t) +
	\partial_w \cdot(V(s_1,s_2,w) \gamma_t)&= 0  \label{eq:conddist2}
\end{align}
Note that we consider $\gamma_t$ as a map from ${\cal S}$ to ${\cal M}_+({\cal S} \times {\cal W})$.
It is straight-forward in this formulation to verify that the equations are consisted with \eqref{eq:charweightconcentration}, i.e. the
weight-concentration case still yields special solutions of \eqref{eq:conddist1}, \eqref{eq:conddist2}. This further leads to the idea to investigate the latter with a similar system of characteristics, which we will discuss in the following.

%
%\begin{align} 
%\partial_t \int \varphi(z_2) \mu^{ 2}_t(dz_2)  =&   
%\sum_{i \leq 2}  \int \nabla_{s_i} \varphi(z_2)   U(s_i,s_{3},w_{i3}) \frac{d\mu_t^{ 2}}{d\mu_t^{ 1}}(ds_3,dw_{i3})  \mu_t^{\infty:2}(dz_2)      +  \nonumber \\ & \int \nabla_{w_{12}} \varphi(z_2)   V(s_1,s_2,w_{12}) \mu_t^{ 2}(dz_2 )  = 0.
%\end{align}

\subsubsection{A method of characteristics} 
 
In order to analyze the existence and uniqueness of solutions to \eqref{eq:conddist1}, \eqref{eq:conddist2}, we study
the characteristic system
\begin{align}
\partial_t S(t;s_1,w) &= \int  U(S (t;s_1,w),S(t;s',w'),W(t;s_1,s',w'))  \gamma_0(s_1;ds',dw') \label{eq:char1} \\
%\partial_t S_2(t;s_1,s_2,w_{12}) &= \int  U(S_2(t;s_1,s_2,w_{12}),S_2(t;s_2,s_3,w_{23}),W(t;s_2,s_3,w_{23})) \gamma_0(ds_3,dw_{13})\\
\partial_t W(t;s_1,s_2,w) &= V(S(t;s_1),S (t;s_2,w),W(t;s_1,s_2,w)), \label{eq:char2}
\end{align}
which generalizes the one in the case of weight concentration. With an analogous proof to the one of Proposition \ref{weightconcentrationproof} 

\begin{prop}
Let $U$ and $V$ be Lipschitz continuous and let $\mu^1_0 \in {\cal M}_+(\R^m), \gamma_0 \in C(\R^m;{\cal M}_+(\R^m \times \R))$
Then there exists a unique solution 
$$ S \in C^1(0,T;C(\R^m)), \qquad W \in C^1(0,T;C(\R^{2m+1}))$$
of the characteristic system \eqref{eq:char1}, \eqref{eq:char2}. Moreover, $(s_1,s_2,w) \mapsto (S(t;s_1,w),S(t;s_2,w),W(t;s_1,s_2,w))$ is a one-to-one map on $\R^{2m+1}$.
\end{prop}
%\begin{proof}
%The fact that $(S,W)$ is injective follows from the uniqueness of the ODE system, the surjectivity from the analogous existence result that can be applied to \eqref{eq:char1}, \eqref{eq:char2} with terminal conditions.
%\end{proof}

Based on the solutions of the characteristic system we can construct a solution of  \eqref{eq:conddist1}, \eqref{eq:conddist2}: we define $\mu^1_t$ as the push-forward of $\mu^1_0$ under $S$ and $\gamma^t$ via
$$ \int \int \varphi(S(t;s_2,w),W(t;s_1,s_2,w)) \gamma_t(S(t;s_1,w);ds_2,dw) = \int \int \varphi(s_2,w) \gamma_0(s_1;ds_2,dw). $$
for all $t \in [0,T]$, $s_1 \in \R^m$ and $\varphi \in C_b(\R^{m+1})$.
\begin{cor}
Let $U$ and $V$ be Lipschitz continuous and let $\mu^1_0 \in {\cal M}_+(\R^m),  \gamma_0 \in C(\R^m;{\cal M}_+(\R^m \times \R))$. Then there exists a unique solution  $\mu_t^1 \in C(0,T;{\cal M}_+(\R^m))$, $\gamma_t \in C(0,T;C(\R^m;{\cal M}_+(\R^m \times \R)))$ of \eqref{eq:conddist1}, \eqref{eq:conddist2}.
\end{cor}

\subsubsection{Propagation of gradient flow structures}

In a gradient flow setting ($U = - \nabla_s F, V= - c \partial_w F$), we can rewrite the energy as
$$ \hat E(\mu^2) =  \int \int \int F(s_1,s_2,w) \mu^2(ds_1,ds_2,dw) . $$
It is thus more convenient to use the formulation in terms of the pair distribution $\mu^2$ instead of \eqref{eq:conddist1}, \eqref{eq:conddist2}. We have
\begin{align*} 
\partial_t  \mu^{ 2}_t   =& \nabla_{s_1} \cdot (\int   \nabla_{s_1} F(s_1,s',w')  \gamma_t(s_1;ds',dw')  \mu^2_t) 
	+  \nonumber \\
& \nabla_{s_2} \cdot (\int   \nabla_{s_2} F(s',s_2,w')  \gamma_t(s_2;ds',dw')  \mu^2_t) +  c	\partial_w \cdot( \partial_w F(s_1,s_2,w) \mu^2_t), 
\end{align*}
which is indeed an abstract gradient flow formulation, since it applies a negative semidefinite operator to the energy variation. The energy dissipation is given by 
\begin{align*} 
\frac{d}{dt} \hat E(\mu_t^2) =&  - \int   \left\vert \int \int \nabla_{s_1} F(s_1,s',w') \gamma_t(s_1,ds',dw') \right\vert^2 \mu_t^1(ds_1) \\& - \int   \left\vert \int \int \nabla_{s_2} F(s_2,s',w') \gamma_t(s_2,ds',dw') \right\vert^2 \mu_t^1(ds_2) \\& - c \int \int \int |\partial_w F(s_1,s_2,s_3|^2 \mu_t^2(ds_1,ds_2,dw). 
\end{align*}
Thus, the gradient flow structure and energy dissipation is propagated through this kind of closure, which appears to be a unique property for this kind of closure. In the general case the first two dissipation terms on the right-hand side would change, e.g. the first one to
$$
- \int  \int    \int \int \int \nabla_{s_1} F(s_1,s_3,w_{13})  \nabla_{s_1} F(s_1,s_2,w_{12})    \mu_t^3(ds_1,ds_2,ds_3,dw_{12} dw_{13},dw_{23}) 
$$
and with other closures such as the Kirkwood closure there is no particular reason for this term to be negative or in other terms the associated bilinear form of $F$ is not definite.

\section{Variants and Open Problems}

\subsection{Weight Interactions and Balance Theory}

In our previous treatment we have assumed that there is no interaction between weights, but this was mainly a simplification to present the basic approach. Many recent treatments of social network formation rely on social balance theory, which effectively implies that the change of network links (weights) depends on the weights in a triangle with a third agent (cf. \cite{belaza,pham,pham2,ravazzi}). In our modelling approach this is a minor modification at the microscopic level to 
\begin{align}
\partial_t \mu_t^N =& \frac{1}N \sum_{i=1}^N \sum_{j\neq i} D_{s_i}^* (U(s_i,s_j,w_{ij}) \mu_t^N) + \frac{1}N\sum_{i=1}^N \sum_{j\neq i} 
\sum_{k\neq i,j}  D_{w_{ij}}^* (V(s_i,s_j,w_{ij},w_{ik},w_{jk}) \mu_t^N) \nonumber \\
& + \sum_{i=1}^N  D_{s_i}^* (U_0(s_i) \mu_t^N) - \sum_{i=1}^N  D_{s_i}^*  D_{s_i} (Q(s_i) \mu_t^N) - \sum_{i=1}^N \sum_{j\neq i} D_{w_{ij}}^*D_{w_{ij}}(R(s_i,s_j,w_{ij}) \mu_t^N). \label{eq:weightinteraction} 
\end{align}
There is however an interesting consequence for the macroscopic description, since we will get an integral involving a higher order moment also in the term with $V$. A simple closure based on the conditional distribution will definitely not suffice if $V$ depends on a triangle of weights, hence the Kirkwood closure seems to be the apparent option. 

\subsection{Beyond Pair Interaction}

As we have seen already in the example of the voter model from \cite{thurner} above, several models go beyond a simple pair interaction, e.g. with triplets (cf. \cite{lambiotte,zanette,neuhaeuser}) or with interactions via the degree (cf. \cite{kohne,raducha,tur}), which effectively averages over weights. Another effect found in some models is an averaging of states of the network or adjacent nodes (cf. e.g. \cite{baumann}). Since many social processes may indeed appear from interaction of more than two agents, it will be an interesting direction to advance the approach to such cases and study their effects on pattern formation. In particular it will be relevant to study whether the pair distribution is still sufficient to provide a suitable approximation.

\subsection{Analysis of Pair Closures}

From a mathematical point of view, an important open challenge is the analysis of pair closures that seem inevitable in the case of network problems. Already the analysis of the arising equations for the pair distribution is not fully covered by the existing PDE theory, it is also open whether it is more suitable to study the formulation with the pair distribution or rather the one with the conditional distribution.
Also the analysis of gradient structures or generalizations at the level of the pair distribution is an interesting topic for further investigations, related to the evolution of gradient structures like in \cite{burger} and possibly with other structures as in \cite{maas} for discrete models. 

An even more challenging question is the analysis of pair closure relations for the BBGKY-type hierarchy, for which one would possibly need similar quantitative estimates as for the mean-field limit in \cite{mischler,paul}, which currently seem out of reach. Another level of complication would be to go beyond the mean-field scaling we have confined our presentation to. 

From a more practical point of view it needs to be investigated which macroscopic models are efficient for simulations and in particular which can be effectively calibrated from data in order to make more than rough qualitative predictions.

\section*{Acknowledgements}

The author acknowledges partial financial support by  European Union’s Horizon 2020 research and
innovation programme under the Marie Sklodowska-Curie grant agreement No. 777826
(NoMADS) and the German Science Foundation (DFG) through CRC TR 154  "Mathematical Modelling, Simulation and Optimization Using the Example of Gas Networks", Subproject C06.

\begin{appendix}

\section{Minimal Model} \label{appendixminimal}

\subsection*{Proof of Theorem \ref{minimalexistencethm}}

Local existence and uniqueness in time is a direct consequence of the Picard-Lindel\"of Theorem and since the ODE systems \eqref{eq:minimal1}-\eqref{eq:minimal6} as well as \eqref{eq:minimalk1}-\eqref{eq:minimalk6} are autonomous with $C^\infty$ right-hand side we obtain the regularity. Nonnegativity of solutions can be inferred easily from standard results on positive semigroups by writing the system in the form $\partial_t u = - A(u) u$, where $A(u)$ is an M-matrix, which implies that the resolvent is positive (cf. \cite{poole}) and hence the
(cf. \cite{arendt}). Then with the conservation of the sum ($\partial_t (\rho_+ + \rho_-)=0$) we see that all components of the solution are contained in $[0,1]$. As long as $\rho_+$ and $\rho_-$ are bounded away from zero, the right-hand side in both ODE systems is globally Lipschitz-continuous, thus a unique solution exists. Hence, a finite existence interval can only occur if $\rho_+(t) \rightarrow 0$ or $\rho_-(t) \rightarrow 0$. In the case  $\alpha_{-+} = \alpha_{-+}$ we see from \eqref{eq:rhoplusevolution} that $\rho_+$ and $\rho_-$ are conserved, hence remain bounded away from zero globally in time.

\subsection*{Linearization around polarized stationary solutions}

We denote the variables for the linearized problem by $\tilde f_{**}$ and $\tilde g_{**}$, respectively. The linearized problem corresponding to \eqref{eq:minimal1}-\eqref{eq:minimal6} around a stationary state with $f_{+-} = 0$ is given by
\begin{align*}
\partial_t \tilde f_{++} =&   - \alpha_{+-} \frac{f_{++} }{\rho_+} \tilde f_{+-} + \beta_{++} \tilde g_{++} - \gamma_{++} \tilde f_{++}
 \\
\partial_t \tilde g_{++} =& \alpha_{-+} \frac{g_{+-}}{\rho_-} \tilde f_{+-} - \alpha_{+-} \frac{g_{++} }{\rho_+} \tilde f_{+-} - \beta_{++} \tilde g_{++} + \gamma_{++} \tilde f_{++} \\
\partial_t \tilde f_{--} =&    - \alpha_{-+} \frac{f_{--} }{\rho_-} \tilde f_{+-} + \beta_{--} \tilde g_{--} - \gamma_{--} \tilde f_{--} \\
\partial_t \tilde g_{--} =& \alpha_{+-} \frac{g_{+-} }{\rho_+} \tilde  f_{+-} - \alpha_{-+} \frac{g_{--}}{\rho_-} \tilde  f_{+-} - \beta_{--} \tilde g_{--} + \gamma_{--} \tilde f_{--} \\
\partial_t \tilde f_{+-} =&  \left( \alpha_{+-} \frac{f_{++}}{2\rho_+}   + \alpha_{-+} \frac{f_{--}  }{2\rho_-} - \gamma_{+-} \right) \tilde f_{+-} \\
\partial_t \tilde g_{+-} =&  \left( - \alpha_{-+} \frac{g_{+-}}{2\rho_-} + \alpha_{+-} \frac{g_{++}  }{2\rho_+}  - \alpha_{+-} \frac{g_{+-}}{2\rho_+} + \alpha_{-+} \frac{g_{--} }{2\rho_-} + \gamma_{+-} \right) \tilde f_{+-} \nonumber  % - \beta_{+-} \tilde g_{+-} 
\end{align*}
The linearized problem corresponding to \eqref{eq:minimalk1}-\eqref{eq:minimalk6} is analogous with additional quotients of the form $\frac{h_{*\diamond}}{\rho_+ \rho_\diamond}$ with $*,\diamond \in \{+,-\}$.

\subsection*{Proof of Proposition \ref{smallepsilonprop}}

We fix $\rho_+$ (and consequently $\rho_- = 1 - \rho_+$) and study the implicit equation around $\epsilon = 0$ and the stationary solution from Proposition \ref{minimalstatprop}. In order to obtain a regular linearized system we first eliminate the variables $g_{++}$ and $g_{--}$
from the given conditions for the fixed densities $\rho_+$ and $\rho_-$ to obtain
$$g_{++} = \rho_+ - f_{+-} - g_{+-} - f_{++}, \quad g_{--} = \rho_- - f_{+-} - g_{+-} - f_{--}. $$
\begin{align*}
 \alpha_{-+} \frac{f_{+-}^2}{\rho_-} - \alpha_{+-} \frac{f_{++} f_{+-}}{\rho_+} - \beta_{++} f_{+-} - \beta_{++} g_{+-} -( \beta_{++} + \gamma_{++}) f_{++} &= - \beta_{++} \rho_+
  \\
   \alpha_{+-} \frac{f_{+-}^2}{\rho_+} - \alpha_{-+} \frac{f_{--} f_{+-}}{\rho_-} - \beta_{--} f_{+-} - \beta_{--} g_{+-} -( \beta_{--} + \gamma_{--}) f_{--} &=  - \beta_{--} \rho_+\\
  - \alpha_{-+} \frac{f_{+-}^2}{2\rho_-} + \alpha_{+-} \frac{f_{++} f_{+-}}{2\rho_+} - \alpha_{+-} \frac{f_{+-}^2}{2\rho_+} + \alpha_{-+} \frac{f_{--} f_{+-}}{2\rho_-}   + \epsilon g_{+-} - \gamma_{+-} f_{+-}  &= 0\\
 - \alpha_{-+} \frac{g_{+-}f_{+-}}{2\rho_-} + \alpha_{+-} \frac{(\rho_+ - f_{+-} - g_{+-} - f_{++}) f_{+-}}{2\rho_+}  - \alpha_{+-} \frac{g_{+-}f_{+-}}{2\rho_+} + \\ \alpha_{-+} \frac{(\rho_- - f_{+-} - g_{+-} - f_{--}) f_{+-}}{2\rho_-} - \epsilon g_{+-} + \gamma_{+-} f_{+-} &=0  .
\end{align*}
It is straightforward to verify that the linearization at $\epsilon = 0$ and the corresponding stationary solution is regular under the above conditions and hence, by the implicit function theorem there exists a locally unique solution for $\epsilon$ sufficiently small, whose dependence on $\epsilon$ is differentiable. From the equation
$$  \left( \alpha_{-+} \frac{f_{+-} }{2\rho_-} -\alpha_{+-} \frac{f_{++} }{2\rho_+} + \alpha_{+-} \frac{f_{+-}}{2\rho_+} - \alpha_{-+} \frac{f_{--}}{2\rho_-}  +  \gamma_{+-}   \right)f_{+-}=  \epsilon g_{+-}  $$
it is straightforward to verify that $\frac{df_{+-}}{d\epsilon}(0)$ is positive, hence the solution is indeed nonnegative for $\epsilon$ sufficiently small. 

\end{appendix} 

\end{document}